\documentclass{amsart}

\usepackage{amsmath, amssymb, amsthm}
\usepackage{hyperref}
\usepackage{tikz}
\usetikzlibrary{shapes,arrows}
\usetikzlibrary{intersections}
\usepackage{color}

\title{A non-computable c.e.~closed subset of $[0,1]$}

\author[S. Badaev]{Serikzhan Badaev}
\address{Kazakh-British Technical University, Almaty, Kazakhstan}
\email{s.badaev@kbtu.kz}

\author[N. Bazhenov]{Nikolay Bazhenov}
\address{Sobolev Institute of Mathematics, Novosibirsk, Russia}
\address{Kazakh-British Technical University, Almaty, Kazakhstan}
\address{Nazarbayev University, Astana, Kazakhstan}
\email{nickbazh@yandex.ru}

\author[S. Goncharov]{Sergey Goncharov}
\address{Sobolev Institute of Mathematics, Novosibirsk, Russia}
\email{s.s.goncharov@math.nsc.ru}

\author[B. Kalmurzayev]{Birzhan Kalmurzayev}
\address{Kazakh-British Technical University, Almaty, Kazakhstan}
\email{birzhan.kalmurzayev@gmail.com}

\author[A. Melnikov]{Alexander Melnikov}
\address{Victoria University of Wellington, New Zealand}
\email{alexander.g.melnikov@gmail.com}

\theoremstyle{plain}
\newtheorem{theorem}{Theorem}[section]
\newtheorem{proposition}[theorem]{Proposition}
\newtheorem{lemma}[theorem]{Lemma}
\newtheorem{corollary}[theorem]{Corollary}

\theoremstyle{definition}
\newtheorem{definition}[theorem]{Definition}

\newtheorem{property}[theorem]{Property}
\newtheorem{notation}[theorem]{Notation}

\theoremstyle{remark}
\newtheorem{remark}[theorem]{Remark}

\usepackage{tikz}
\usepackage{cancel}
\usepackage[all]{xy}
\usepackage{tikz}
\usetikzlibrary{arrows}
\usetikzlibrary{decorations.markings}
\usetikzlibrary{decorations.pathreplacing}
\usetikzlibrary{patterns, calc}
\usetikzlibrary{arrows.meta, positioning, backgrounds}
\usetikzlibrary{shapes.geometric}
\usetikzlibrary{matrix}
\usepackage{tkz-graph}
\usepackage{tikz-cd}

\usepackage{wrapfig}
\usepackage{stackengine,scalerel}

\begin{document}

\begin{abstract} We prove that there exists a $\Sigma^0_1$ closed subset of $[0,1]$
that is not homeomorphic to any computably compact space.
We show that the index set of c.e.~subspaces of $[0,1]$ that admit 
a computably compact presentation is not arithmetical, as witnessed by subsets of $[0,1]$.
The index set result is new for computable Polish spaces in general, not only for those realised as c.e.~closed subsets of $[0,1]$.
\end{abstract}

\maketitle

\tableofcontents


\section{Introduction}
The central focus of this paper is the class of compact Polish spaces. The most commonly used notions of algorithmic presentability for such spaces are:
\begin{enumerate}
\item computable Polish presentations;
\item computably compact presentations;
\item $\Pi^0_1$ and $\Sigma^0_1$~closed presentations.
\end{enumerate}

We postpone formal definitions of these standard notions to the preliminaries. 
Informally: (1)~requires a countable dense subset with a computable metric; (2)~strengthens this by demanding an effective, uniform refinement of any cover to a finite sub-cover; and (3)~fixes an ambient space \( M \) (e.g., \( \mathbb{R}^n \), \( (C[0,1]; \mathbb{R}) \), or the Hilbert Cube) and requires either an effectively enumerable complement (\( \Pi^0_1 \)-closed) or an effectively enumerable  sequence (within \( M \)) that is dense in the closed set (\( \Sigma^0_1 \)-closed, respectively).

We view our spaces, and their presentations, up to homeomorphism. For instance, a space \( C \) has a computably compact presentation if it is \emph{homeomorphic} to a computably compact space. It is known that only trivial implications hold between (1), (2), and (3) up to homeomorphism. For example, being computable Polish is equivalent to being c.e.~closed (this follows from, e.g., the main result of~\cite{UnivC01}).  For further details, we refer to the recent surveys~\cite{IlKi,effcomp}. 
However, we note that the results distinguishing (1), (2), and (3) are very recent  and have been established in \cite{Bosshomo,uptohom,topsel,lupini,bastone}. This contrasts with the situation in effective algebra, where standard notions of effective presentability, such as c.e., computable, \( n \)-decidable, and decidable structures, were separated nearly half a century ago (e.g., \cite{Khi,feiner1970,Hig}). 

\subsection{Computably compact presentations}
The notion of a computably compact Polish space is particularly robust, admitting many equivalent formulations. For instance, a compact computable Polish space is computably compact if and only if its continuous diagram is decidable~\cite{effcomp}. This robustness suggests that computably compact presentations are especially desirable and, as argued in~\cite{effcomp}, could serve as a foundation for the theory of effectively presented compact Polish spaces.
In contrast, the notions of computable Polish and $\Pi^0_1$ and $\Sigma^0_1$~closed presentations are weaker up to homeomorphism, in general (this follows from, e.g.,~\cite{uptohom, bastone}).
One of the central questions of such studies is:

\begin{center}Q1: \emph{Which $\Pi^0_1$, $\Sigma^0_1$ compacta are homeomorphic to  computably compact ones?}
\end{center}

Remarkably little is known even for relatively tame classes of spaces. We briefly summarise the current state of knowledge.
The first examples of (compact) \( \Pi^0_1 \), \( \Sigma^0_1 \) closed sets that are not homeomorphic to computably compact spaces appear in~\cite{uptohom, topsel}. Regarding positive results, it is known that every computable (equivalently, c.e.~closed) Stone space is homeomorphic to a computably compact one~\cite{uptohom, effcomp}. In contrast, there exists a Stone space that is \( \Pi^0_1 \)-closed in \([0,1]\) but not homeomorphic to any computably compact space~\cite{bastone}.
A complete characterisation of computably compact solenoid spaces has been established~\cite{lupini, pontr}, as well as, more generally, the connected Pontryagin duals of discrete countable abelian groups. Similarly, a broad class of solenoid spaces admitting c.e.~closed presentations, up to homeomorphism, has been characterised~\cite{newpolish}. 
Additional results include descriptions of specific ad hoc classes of spaces, such as star-spaces (wedge sums of finitely many intervals) and disjoint sums of \( n \)-spheres, found in~\cite{topsel, uptohom, CounterExa}.
Further interesting closed subsets of $\mathbb{R}^n$ were constructed in \cite{Bosshomo}, however, the results in \cite{Bosshomo} work only up to self-homeomorphism of the ambient space. We also cite the somewhat related works~\cite{MillerBalls, mani1, mani2, mani3, mani4, AmirHoyrup2023}, which, however, typically consider a \emph{fixed} \( \Pi^0_1 \)-closed subset of \( \mathbb{R}^n \) that is homeomorphic to, for example, an \( m \)-sphere or an arc. The question then posed is whether this fixed closed set is also $\Sigma^0_1$ (and therefore computably compact). This framework is almost exclusively focused on positive results.

In summary, nearly all interesting counterexamples to date are derived either from effective Pontryagin or Stone duality or are artificially constructed using wedge sums and disjoint sums of intervals and \( n \)-spheres. This briefly summarises our current knowledge and technical tools in the field.

Given that many  counterexamples are constructed using \( n \)-stars, \( n \)-spheres (\( n \geq 2 \)), and solenoid spaces, one might wonder whether  interesting complex examples can be found within the unit interval. (None of these spaces can be embedded into \([0,1]\).) 
We have seen that there exist \( \Pi^0_1 \)-closed, totally disconnected (Stone) subsets of \([0,1]\) that are not homeomorphic to any computably compact space. On the other hand, it is known that c.e.~Stone spaces admit computably compact presentations. 
The following question, therefore, was left open in~\cite{CounterExa}:

\begin{center}
Q2: \emph{Is there a c.e.~closed $C \subseteq [0,1]$ with no computably compact homeomorphic copy?}
\end{center}

We remark that partial results were proven in~\cite{Bosshomo}; however, as noted earlier, these results work only up to self-homeomorphism of \([0,1]\), and not up to homeomorphism in general\footnote{While the article was in production, Bosserhoff [A counterexample regarding c.e.\ closed subsets of $[0,1]$ under homeomorphisms, arXiv preprint arXiv:2508.00166, 2025] announced an independent solution to Q2.}.

\subsection{Results}
We shall attack both Q1 and Q2 simultaneously by proposing a new method for constructing c.e.~closed subsets of \([0,1]\) that resemble Stone spaces. This approach relies heavily on the theory of computable Boolean algebras. For example, we will utilise Thurber's well-known result~\cite{ThurThesis}, which states that every \( \mathrm{low}_2 \) Boolean algebra is isomorphic to a computable one.

The results are as follows.

\begin{theorem}\label{thm:1} There exists a c.e.~closed subset $C$ of the standard presentation of the unit interval $[0,1]$ 
so that $C$ is \emph{not} homeomorphic to any computably compact space.
   \end{theorem}
   
   We state here one almost immediate corollary which, in a way, improves \cite[Theorem 1.3]{CounterExa}:

\begin{corollary}\label{cor:Ba}
There exists a closed subset $C \subseteq [0,1]$ so that the Banach space $(\mathcal{C}[C; \mathbb{R}], d_{sup}, +)$ has a computable presentation, but $C$ has no computably compact homeomorphic copy.
\end{corollary}

In \cite{CounterExa}, a c.e.~closed subset of $[0,1]^2$ with this property is constructed. As explained in \cite{CounterExa}, the significance of this result is that the effective version of Banach--Stone Duality between 
 $\mathcal{C}[C; \mathbb{R}]$ and compact $C$ \emph{fails}, at least in the strongest desirable form. This answered a question of McNicholl.
 In contrast, the Banach--Stone duality is effective for Stone spaces~\cite{bastone}, and the closely related Gelfand duality (for $C^*$-algebras of the form $\mathcal{C}[C; \mathbb{R}]$) is also effective in general~\cite{burton2024computablegelfandduality}.
 The reader may find it somewhat surprising that a counterexample to the effective Banach--Stone duality can be found already among closed subsets of $[0,1]^2$, let alone $[0,1]$.

\
   
   We now discuss the second result.
   Fix a uniform enumeration $(M_i)_{i \in \omega}$ of all (partial) computable Polish spaces. (We remark that, for $M_i$, the property of being compact is $\Pi^0_3$~\cite{CompComp}.)

\begin{theorem}\label{thm:2}  The index set 
$$\{i: M_i \mbox{ is homeomorphic to a computably compact space}\}$$
is not arithmetical, i.e., not $\Sigma^0_n$ for any $n \in \omega$.
   \end{theorem}
   
The stated earlier Theorem~\ref{thm:1} resolves Q2 from~\cite{CounterExa}, while Theorem~\ref{thm:2} uses tools from the proof of Theorem~\ref{thm:1} to provide a lower estimate for the index set complexity in Q1. This lower estimate is the best known to date for spaces in general, not just for closed subsets of \([0,1]\). Our techniques seem to be limited to arithmetical sets, but we conjecture that the index set in Theorem~\ref{thm:2} is $\Sigma^1_1$-complete\footnote{It is $\Sigma^1_1$ as follows, e.g., from \cite[Proposition 8.1.34]{TheBook}
and the fact that every computably compact space can be realised as a computable subset of the Hilbert Cube; see Proposition~\ref{prop:stuff}.}. 

Although our proofs are not particularly difficult, they rely on new ideas and relatively advanced tools. For example, we analyse the proof of Feiner's classical theorem~\cite{feiner1970}, apply the aforementioned theorem of Thurber~\cite{ThurThesis}, and prove an extension of the Remmel--Vaught Theorem~\cite{Vau,RemIso} for a class of algebras extending Boolean algebras. Consequently, much of the intrinsically complex combinatorics is delegated to both new and classical results in effective algebra (and their proofs).
Our proofs also suggest that an effective duality might exist between Boolean algebras with a distinguished subset of atoms and certain closed subsets of \([0,1]\). In this article, we shall prove only  partial effective duality-type results sufficient to derive the main theorems, leaving the verification of the general effective duality for future investigation.

\section{Preliminaries}
\subsection{Computable Polish spaces and effectively closed sets}
\begin{definition}
A  \emph{computable} (presentation of a) \emph{Polish space} is given by:
(1) a~dense sequence $(x_i)_{i \in \omega}$, perhaps with repetitions, and (2)~ 
a~computable function $f$ which, given $i, j, s \in \omega$,
outputs $r = \dfrac{n}{m} \in \mathbb{Q}$ such that
$|d(x_i, x_j) - r| < 2^{-s},$
where $d$ is the metric on the space. 
\end{definition}

We view spaces up to homeomorphism, and thus we require that the metric is compatible with the topology.
We also require that the metric is complete, and so $M =  \overline{(x_i)_{i \in \omega}}$. This is not really a restriction in the compact case, since
every compact metric space is necessarily complete.

Points $x_j$ from this sequence are called \emph{special, ideal,} or (less frequently) \emph{rational}.  
A \emph{basic open ball} is a ball of the form $B(x_{j}, r) = \{y \in M: d(x_j, y) < r \}$.
Here, $x_j$ is a special point and $r \in \mathbb{Q}$ is positive. We also always represent rational numbers 
as fractions when possible. In particular, a basic open ball is assumed to have its radius represented as a fraction.
We say that an open set $V$ is c.e.~(in a computable Polish space $M$) if $V$ is a c.e.~union of basic open balls represented in this way.
Say that a sequence of special points $(y_j)_{j \in \omega}$ is \emph{fast Cauchy} if $d(y_j, y_{j+1}) < 2^{-j}$, for all $j$.
The \emph{name} of a point $x \in M$ of a computable Polish space $M$ is the set $N^x = \{B \ni x : B \mbox{ is basic open}\}$.
A point is computable if it admits a c.e.~ name. (Equivalently, if there exists a computable fast Cauchy sequence converging to the point.)

\begin{definition} Let $M$ be a computable Polish space and assume $C \subseteq M$ is closed. Then:
\begin{enumerate}
\item $C$ is $\Pi^0_1$ or \emph{effectively closed} if $(M - C)$ is c.e.~open.
\item  $C$ is $\Sigma^0_1$ or \emph{c.e.~closed} if there is a uniformly computable sequence $(y_j)_{j \in \omega}$
of $M$-computable points that is dense in $C$, i.e., $C = \overline{(y_j)_{j \in \omega} }$.
\item \emph{computable closed} if it is both $\Pi^0_1$ and $\Sigma^0_1$.
\end{enumerate}

\end{definition}

It is clear that, in (2) above, the sequence $(y_j)_{j \in \omega}$ together with the metric in $M$ induces a computable Polish presentation of $C$.
Conversely, every computable Polish space can be computably isometrically embedded into $C[0,1]$ or the Urysohn space~\cite{UnivC01}.
Thus, even up to isometry, computable Polish and c.e.~closed presentations are equivalent.

\subsection{Computable Banach spaces}
To state the next definition, we need the notion of a (Type II) computable function.

\begin{definition}
Let $f\colon M \to N$ be a function between two computable Polish spaces. We say that $f$ is (Type II) \emph{computable} 
if it uniformly effectively turns fast Cauchy sequences into fast Cauchy sequences. More formally,
if there is a uniform sequence of operators $(\Phi_n)_{n \in \omega}$ such that on input a sequence $(x_{i_n})_{n \in \omega}$ with $d_M(x_{i_n}, x_{i_{n+1}}) < 2^{-n}$,
we have that $f( \lim_n x_{i_n}) = \lim_m \Phi^{(x_{i_n})_{n \in \omega}}_m$
and
$$d_N(\Phi^{(x_{i_n})_{n \in \omega}}_m, \Phi^{(x_{i_n})_{n \in \omega}}_{m+1}) < 2^{-m},$$
for all $m$.

\end{definition}

\begin{definition}\label{def:basp}
A computable (real) Banach space is given by a computable Polish space for the metric induced by the norm upon which the operation $+$ is computable.
\end{definition}

As explained in \cite{MelNg}, this approach is equivalent to the other standard definitions found throughout the literature. (See also~\cite[Lemma 2.4.17]{TheBook}.)

\subsection{Computably compact spaces} The definition below is equivalent to the one informally stated in the introduction.
\begin{definition}[\cite{MoriTsujiYasugi}]\label{def:1}
A computable Polish space $M$ upon a dense set $(x_i)_{i \in \omega}$ is \emph{computably compact} if there is a computable function which, given $n$,
outputs a finite tuple $i_0, \ldots, i_k $ of natural numbers such that
$M = B(x_{i_0}, 2^{-n}) \cup \ldots \cup B(x_{i_k}, 2^{-n}),$
i.e., it is a finite (basic) open $2^{-n}$-cover of the space.
\end{definition}

\begin{proposition}[Folklore combined with \cite{effcomp}] \label{prop:stuff} \rm
For a computable Polish space $M = \overline{(x_i)_{i \in \omega}}$, the following are equivalent:
\begin{enumerate}
\item $M$ is computably compact (Definition~\ref{def:1}).


\item Similar to Definition~\ref{def:1}, but we can additionally decide whether a given $2^{-n}$-ball from the $2^{-n}$-cover intersects
any $2^{-m}$-ball from the $2^{-m}$-cover (including the case when $m =n$). Additionally, two basic closed balls from the covers intersect if, and only if, 
the respective basic open balls intersect.

\item There is a computably enumerable list of \emph{all} finite open covers of the space by basic open balls.

\item There is an effective procedure which, given an enumeration of a countable cover of the space by basic open balls, outputs a finite sub-cover.

\item There is a computable function $h\colon \mathbb{N} \to \mathbb{N}$ such that
$M = \bigcup_{i \leq h(n)} B(x_i, 2^{-n}).$

\item $M$ is computably homeomorphic to a computably closed subset of the standard computable presentation of the Hilbert Cube.
\end{enumerate}
\end{proposition}

See \cite{effcomp} for a proof. For many more equivalent formulations of computable compactness, we cite \cite{effcomp}.
We will use the following  elementary consequence of
Proposition~\ref{prop:stuff}(2) that can also be found in \cite{effcomp}, though in a slightly different form.

\begin{lemma}\label{lem:Dlist} Let $D$ be a computably compact space. There exists a uniformly effective enumeration 
of all clopen partitions of $D$, where each clopen set is represented as a finite tuple of basic open balls or, equivalently, by the tuple 
describing the basic open balls with the same parameters.
Furthermore, we can uniformly decide which clopen sets in the list intersect and which don't.

\end{lemma}



\section{Proofs}

We begin with the proof of Theorem~\ref{thm:1}. Our proof of Theorem~\ref{thm:2} will rely on the techniques developed in the proof of Theorem~\ref{thm:1}.

\subsection{Feiner's Boolean algebra}
A \emph{c.e.\ presentation} of a Boolean algebra $\mathcal{B}$ is a pair $(\mathcal{C},E)$, where  $\mathcal{C}$ is a computable structure in the signature of Boolean algebras, $E$ is a c.e.\ congruence of the structure $\mathcal{C}$, and the quotient structure $\mathcal{C}/E$ is isomorphic to the algebra $\mathcal{B}$. Equivalently, a c.e.~presentation can be thought of as the quotient of a fixed natural computable presentation of the atomless Boolean algebra by a c.e.~ideal. A \emph{$\Delta^0_2$-presentation} of a Boolean algebra $\mathcal{B}$ is a $\Delta^0_2$-computable structure isomorphic to $\mathcal{B}$. 

For a linear order $\mathcal{L}$, the corresponding \emph{interval Boolean algebra} $\operatorname{IntAlg}(\mathcal{L})$ is the subalgebra of the Boolean algebra of all subsets of $\mathcal{L}$ such that $\operatorname{IntAlg}(\mathcal{L})$ is generated by the intervals of the form $[a,b)$ and $[a,+\infty)$, where $a<_{\mathcal{L}} b$.

Fix a c.e.~presented Boolean algebra with no computable presentation, \cite{feiner1970}.
We will follow the proof presented in \cite{TheBook}, which is however not really different from the original proof of Feiner.
We won't need to analyse the proof in full detail; instead, we only discuss a few features of this well-known construction
which will be needed in our proof, and we shall take the rest for granted.
We fix a $\Delta^0_2$-presentation (equivalently, a c.e.~presentation) of a Boolean algebra which has the form
$$B  = \sum_{i\in\omega} B_i,$$
where each $B_i$ is either $\operatorname{IntAlg}(\mathbb{Z}^n \mathbb{Q} +1 +\mathbb{Q})$ or 
$\operatorname{IntAlg}(\mathbb{Z}^{n+1} +1 +\mathbb{Q})$ for some $n\geq 1$.
(We could've used  $\omega^n \mathbb{Q}$ vs.~$\omega^{n+1} $ instead.) The linear orders are uniformly $\Delta^0_2$-presented.
We shall need that, furthermore,
\begin{equation}\label{eq:1} \mbox{the extra separating intervals } 1 + \mathbb{Q} \mbox{ in each case are (uniformly) computable.}
\end{equation}
This is possible because these separators are set ahead of time in the construction, and the diagonalisation
happens at the $\mathbb{Z}^n \mathbb{Q}$ vs.~$\mathbb{Z}^{n+1} $ components.  We do not need to know any further details of Feiner's construction; we refer to \cite{TheBook} for a detailed and self-contained exposition of the proof.

We shall represent $B$ using tree-bases, in the sense of \cite{GonBoo}. More specifically, we  represent $B$ as a $\Delta^0_2$-tree with $\Pi^0_2$-atoms/terminal nodes.
Without loss of generality, we can instead realise $B$ as a $\Pi^0_1$-tree $T \subseteq \omega^{<\omega}$ with $\Pi^0_2$-atoms/terminal nodes\footnote{This is a standard trick, which we omit. The idea is to realise $T$ in $\omega^{<\omega}$. We will approximate $T$ using the Limit Lemma and
mimic the process in $\Gamma \subseteq \omega^{<\omega}$. If a node leaves the tree $T$ but is later put back into $T$, instead of reintroducing the node in $\Gamma$, pick a fresh node in $\omega^{<\omega}$. By induction, $T \cong \Gamma$. This makes the tree $\Pi^0_1$; however, the property of being a terminal node remains $\Pi^0_2$.}.
Indeed, by property (\ref{eq:1}), we can assume that:
\begin{equation}\label{eq:2} \mbox{the intervals } 1 + \mathbb{Q} \mbox{ uniformly correspond to computable subtrees in $T$.}
\end{equation}
As before, the reason for that is pretty straightforward~--- the interesting non-computable properties of the tree will be concentrated at the spots where the diagonalisation happens\footnote{In other words, in the notation of the previous footnote, there is no need to imitate the Limit Lemma in the subtree of $\Gamma$ that corresponds to these computable subtrees.}.

Another way to view this property is that the tree $T$ is a ``fishbone'' of trees $T_i$ (see Fig.~\ref{Fig: fish}), where
each $T_i$ is either a complete binary tree (denoted as $\boxed{\eta}$) or one of the trees used in diagonalisation (empty box), and \emph{we can compute the nodes that serve as their roots}.
(There roots are nodes of the form $0^{i}1$.)

\begin{figure}[h]
	
	\begin{tikzpicture}
		[level 1/.style={sibling distance=20mm},
		level 2/.style={sibling distance=20mm},
		level distance=10mm]
		\node[fill, circle, scale=0.5] {}
		child {[fill] circle (2pt)
			child {[fill] circle (2pt)
				child {[fill] circle (2pt)
					child {[fill] circle (2pt)}
					child {node[draw, rectangle, minimum size=4mm, inner sep=0pt] {\(\eta\)}}
				}
				child {node[draw, rectangle, minimum size=4mm, inner sep=0pt] {}}
			}
			child {node[draw, rectangle, minimum size=4mm, inner sep=0pt] {\(\eta\)}}
		}
		child {node[draw, rectangle, minimum size=4mm, inner sep=0pt] {}};
		
		\node at (-4,-4.5) {\(\ldots\)};
	\end{tikzpicture}
	\caption{The tree $T$.}\label{Fig: fish}
	
\end{figure}

\

We assume that $T$ is \emph{binary}, and furthermore, at every stage $s$ of $\Pi^0_1$-approxima\-tion $T_s$ of $T$ and each $\sigma \in T_s$,
\begin{equation}\label{eq:3} \sigma \mbox{ is either terminal or has exactly two children $\sigma \hat{\empty}(2m+1)$ and $\sigma \hat{\empty}(2m+2)$,}
\end{equation}
 for some $m \geq 0$. (The nodes of the form $\sigma \hat{\empty}0$ are reserved for a different purpose.)

If the children of $\sigma$ are removed from $T$, they will never appear again in the approximation. 
However, in this case they will be immediately replaced with another potential pair of children, unless $\sigma$ itself leaves $T$.

\

The plan is to use this tree representation to turn the $\Delta^0_2$ Boolean algebra $B$ into a c.e.~closed subset $C(B)$ of $[0,1]$, and then argue that this c.e.~closed set
cannot have a computably compact presentation.

\subsection{Turning $B$ into a subset of $[0,1]$}
First, we define a sequence of disjoint closed subintervals of $[0,1]$:
$$  I_{0}, I_{2}, \ldots, I_{2n}, \ldots \big( \mbox{concentration point}~\dfrac{1}{2} \big) \ldots, I_{2n+1}, \ldots, I_{3}, I_{1},$$
so that the length of each $I_{m}$ is equal to (say) $2^{-m-2}$. We associate $\langle m \rangle  \in \omega^{<\omega}$ with $I_{m}$.
 Iterate this by scaling the described above pattern in $[0,1]$ by $2^{-n-2}$ and copying it into  $I_{\sigma}$.
To define $I_{\sigma\hat{\empty}m}$, $ m \in \omega$, use the $m^{th}$ sub-interval of $I_{\sigma}$.

Each interval will be associated with a node of $\omega^{<\omega}$, while the central limit point won't be associated with any node.
Indeed, we associate an interval with the entire clopen subset $[\sigma]$ of $\omega^{\omega}$ consisting of all extensions of $\sigma$.

We are now ready to describe the construction of the closed set $C = C(T)$.
Before we proceed,  recall Property (\ref{eq:2}). Recall also that $T$ splits into a sequence of subtrees $T_i$, where we can effectively see which 
of these subtrees are complete binary trees. (This is true even though the tree itself is merely $\Pi^0_1$.)
The construction is split into two sub-constructions.
\subsubsection{The full binary tree sub-construction} If $T_i$ is the full binary tree, and $\sigma = 0^i 1$ is its root,  then uniformly turn interval $I_{\sigma}$ into
a (computably compact) standard copy of the Cantor middle third set.

\subsubsection{The atomic  sub-construction} Assume $T_i$ is not a full binary tree. We shall abuse our notation slightly and will write $T$ instead of $T_i$.
However, the reader should keep in mind that the procedure described below is restricted only to the nodes in $T_i$ and to the respective intervals.

\

Since we will need to place points in the intervals, we need to describe where to place them. This is done as follows:
\begin{equation}
\mbox{If $\sigma$ is terminal in $T_s$ and $I_\sigma \cap C = \emptyset$, then place a  point in $I_{\sigma \hat{\empty}0}\cap C$.}
\end{equation}
We can use any rational point to perform this action; its (strong) index can be calculated uniformly in $\sigma$.
 The two children of each non-terminal $\sigma \in T_s$ 
will be associated with intervals of the form  $I_{\sigma\hat{\empty}(2m+1)}$ and $I_{\sigma \hat{\empty}(2m+2)}$, respectively; see (\ref{eq:3}) and Fig.~\ref{Fig:children}. 
In particular, $I_{\sigma \hat{\empty}0}$ can be used in the construction at most once, and only to put a point into it.

\

\begin{figure}[h]
	\begin{tikzpicture}[level 1/.style={sibling distance=20mm},
		level 2/.style={sibling distance=20mm},
		level distance=10mm]
		
		\node {$\sigma$}
		child {node {$\sigma\hat{\empty}(2m+1)$}}
		child {node {$\sigma \hat{\empty}(2m+2)$}}
		;
		
		\draw (-2.5,-2) -- (2.5,-2);
		
		\node at (-2.5,-2) {\Large$[$};
		\node at (2.5,-2) {\Large$]$};
		
		\node at (2.7, -2.3) {$I_{\sigma}$};
		
		\node at (-2,-2) {$[$};
		\node at (-1,-2) {$]$};
		
		\node at (2,-2) {$]$};
		\node at (1,-2) {$[$};
		
		\node[circle, fill, scale=0.3] at (-1.7,-2) {};
		\node[circle, fill, scale=0.3] at (-1.3,-2) {};
		
		\node[circle, fill, scale=0.3] at (1.3,-2) {};
		\node[circle, fill, scale=0.3] at (1.7,-2) {};
		
		\node at (-1.5,-2.5) {$I_{\sigma \hat{\empty}(2m+1)}$};
		\node at (1.5,-2.5) {$I_{\sigma \hat{\empty}(2m+2)}$};
		
	\end{tikzpicture}
	\caption{Construction for a non-terminal $\sigma$.}\label{Fig:children}
\end{figure}
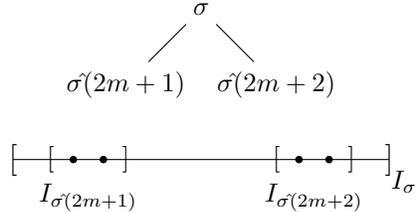

We will also need to densify some parts of $I_\sigma$, so that if $\sigma$ is an atom in $T$ (this is $\Pi^0_2$), then we end up with $I_\sigma \cong [0,1]$.
For that, suppose  $\sigma_1, \sigma_2 \in T_s$ are new successors of $\sigma \in T_s$, 
meaning that a different pair of nodes $\tau_{1}, \tau_{2}$ served as successors of $\sigma$ in $T_{s-1}$. 
(This excludes the case when $\sigma_1, \sigma_2 \in T_s$ are the first-ever successors of $\sigma \in T_{s-1}$ that must appear in $T_s$, according to (\ref{eq:3}).)
In this case, 
\begin{equation}\label{eq:5}
\mbox{begin densifying $I_\sigma$ to the left of $I_{\sigma_1}$ and to the right of $I_{\sigma_2} $,}
\end{equation}
which is done by placing another point between any pair of points already enumerated into $C \cap I_\sigma$ which are  outside of the intervals $I_{\sigma_1}$ and $I_{\sigma_2} $ (see Fig.~\ref{Fig:densify} below).
(We can use rational points in the respective intervals.)

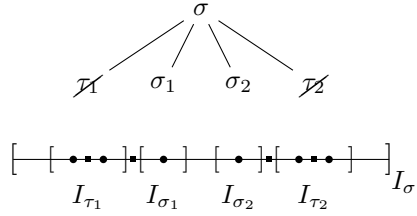
\begin{figure}[h]
	\begin{tikzpicture}[level 1/.style={sibling distance=10mm},
		level 2/.style={sibling distance=10mm},
		level distance=10mm]
		
		\node {$\sigma$}
		child {node {$\cancel{\tau_{1}}$}}
		child {node {$\sigma_{1}$}}
		child {node {$\sigma_{2}$}}
		child {node {$\cancel{\tau_{2}}$}}
		;
		
		\draw (-2.5,-2) -- (2.5,-2);
		
		\node at (-2.5,-2) {\Large$[$};
		\node at (2.5,-2) {\Large$]$};
		
		\node at (2.7, -2.3) {$I_{\sigma}$};
		
		\node at (-2,-2) {$[$};
		\node at (-1,-2) {$]$};
		
		\node at (2,-2) {$]$};
		\node at (1,-2) {$[$};
		
		\node at (-1.5,-2.5) {$I_{\tau_{1}}$};
		\node at (1.5,-2.5) {$I_{\tau_{2}}$};
		
		\node at (-0.8,-2) {$[$};
		\node at (-0.2,-2) {$]$};
		
		\node at (0.2,-2) {$[$};
		\node at (0.8,-2) {$]$};
		
		\node at (-0.5,-2.5) {$I_{\sigma_{1}}$};
		\node at (0.5,-2.5) {$I_{\sigma_{2}}$};
		
		\node[circle, fill, scale=0.3] at (-1.7,-2) {};
		\node[circle, fill, scale=0.3] at (-1.3,-2) {};
		
		\node[circle, fill, scale=0.3] at (1.3,-2) {};
		\node[circle, fill, scale=0.3] at (1.7,-2) {};
		
		\node[circle, fill, scale=0.3] at (-0.5,-2) {};
		\node[circle, fill, scale=0.3] at (0.5,-2) {};
		
		\node[rectangle, fill, scale=0.3] at (-1.5,-2) {};
		\node[rectangle, fill, scale=0.3] at (-0.9,-2) {};
		\node[rectangle, fill, scale=0.3] at (1.5,-2) {};
		\node[rectangle, fill, scale=0.3] at (0.9,-2) {};
		
	\end{tikzpicture}
	\caption{Densifying intervals.}\label{Fig:densify}
\end{figure}

Note that if $\sigma$ is not terminal in $T$, then this action will result in only finitely many points placed outside of $I_{\sigma_1} \cup I_{\sigma_2}$ (i.e., in $I_\sigma  - (I_{\sigma_1} \cup I_{\sigma_2} )$), and none of these extra points will be
between  $I_{\sigma_1}$ and  $I_{\sigma_2}$.  Otherwise, if $\sigma$ is terminal in $T$, we will end up splitting $\sigma$ into a pair of successors over and over again.
Since the middle of $I_\sigma$ is the limit point of $I_{\sigma\hat{\empty}i}$, $i \in \omega$, we shall end up with $C\cap I_{\sigma}$ being the whole of $I_\sigma$ (after we take the closure).

\

This completes the description of $C$. Clearly, $C$ is a c.e.~closed subset of $[0,1]$. 

\subsection{Properties of $C$} We summarise several properties of $C$ that should be clear from the construction.
We begin with the Cantor-like clopen components.
\begin{property}\label{pr:at}
The atomless parts of $B$ (effectively) correspond to clopen copies of the Cantor set. 
\end{property}
This property holds because we did not use any dynamic approximation on the atomless parts, they were handled uniformly by a separate sub-construction.

\

Any clopen connected subset of $C$ is either an isolated point or is a copy of the unit interval.
We now analyse the clopen components that are homeomorphic to~$[0,1]$. 

\begin{definition} 
A clopen subset of $C$ homeomorphic to $[0,1]$ will be called an \emph{intom}.
\end{definition}

The intuition is that, of course, intoms represent the atoms of the Boolean algebra~$B$; however, they are intervals.
Intoms of $C$ are evidently stable under taking the Cantor--Bendixson (CB) derivative (i.e., any intom $D$ is contained in the $\alpha^{th}$ CB derivative $C^{(\alpha)}$ of $C$). However, there exist closed subsets $D$ of $[0,1]$
that contain intoms that are not clopen in $D$: for example, consider a sub-interval in which one (or both) points are limit-points 
of isolated points. One can also think of a closed sub-interval whose boundary is a limit of 
intoms.
Our $C$ is of course not like that:
 
 \begin{property}\label{pr:in}
 If $C'$ is the CB-derivative of $C$, then no maximal path-connected component $I$ in $C'$ is a limit point of a sequence $(x_i)$  with  $x_i \in C - I$.
\end{property}

In other words, no interval-like maximal path-connected component of $C$ can be a limit of isolated points.
This property follows from the fact that every ``densified'' subinterval has the form $I_\sigma$ for some $\sigma$; see (\ref{eq:5}). Thus, it will be kept 
isolated from the rest of $C$, by the choice of the collection $(I_\tau)_{\tau \in \omega^{<\omega}}$.
The next property is certainly less evident, and it requires a more detailed proof.


\begin{lemma}\label{lem:1} We have that $C '' = C'$. (There is no clopen subcomponent of $C$ isomorphic to $\operatorname{IntAlg}(\omega)$.)

\end{lemma}

\begin{proof}

Suppose $x\in C$ is a limit point, but not a limit point in $C'$. (I.e., suppose $x \in C' - C''$.) Then  $x$ cannot be a part of an intom, and it cannot lie in a copy of the Cantor set.
 If $x$ in not in the Cantor-like interval, then it is contained in a clopen component produced in the second sub-contruction
corresponding to either  $\operatorname{IntAlg}(\mathbb{Z}^n \mathbb{Q})$ or $\operatorname{IntAlg}(\mathbb{Z}^{n+1})$, depending on the respective outcome of the Feiner's procedure.
Then $x$ is a limit of a sequence $(x_i)_{i\in\omega}$, where all $x_i$ are isolated, and indeed, every sufficiently small (closed) interval $J \ni x$ contains \emph{only} isolated points, and infinitely many of those.

However, in the construction, any interval $I_\sigma$ contains an isolated point if, and only if, $I_\sigma$ contains $I_{\sigma_1}$ and $I_{\sigma_2}$, where
 $\sigma_1$ and $\sigma_2$ are children of $\sigma$ (in $T$). (In other words, this can happen iff $I_\sigma$ corresponds to an element of the Boolean algebra $B$ that splits non-trivially.)
In this case there are only finitely many atoms in $I_\sigma - (I_{\sigma_1}\cup I_{\sigma_2})$; see the discussion after (\ref{eq:5}).

By our assumption about $x$, it should be that $J$ also contains points in $I_{\sigma_1}$ or $I_{\sigma_2}$, say, the former. Repeat the argument above for $I_{\sigma_1}$,
we get that that $\sigma_1$ also has to split in $T$, and that $J$ must intersect one of the two intervals that represent these children.
Since $J$ was closed, we get that $J$ intersects the limit point of such intervals, none of which represent an atom of the Boolean algebra. In particular, none of these
sub-intervals is an intom. However, both algebras $\operatorname{IntAlg}(\mathbb{Z}^n \mathbb{Q})$ and $\operatorname{IntAlg}(\mathbb{Z}^{n+1})$ are atomic. Thus, every infinite path through  the tree $T$ 
contains arbitrarily long initial segments common with a terminal node in $T$.
Passing to $C$, this implies that $J$ must contain a limit point of intoms. Since $J$ was arbitrarily small, this makes $x$ a limit point of intoms.
This contradicts the choice of $x \in C' - C''$.
\end{proof}

We now summarise the properties and the lemma above in an informal description of $C'$. The description will be made formal using a certain operator, which we will define shortly in the next subsection.

By Lemma~\ref{lem:1}, we have that $C'$ has no isolated points. Further, by Property~\ref{pr:in}, every maximal path-connected component of $C'$ is indeed an intom. The intuition is that
the atoms are the `junk' produced by the construction of $C = C(T)$ from the $\Pi^0_1$ tree $T$, while whatever is `coded' by the intoms and the atomless parts actually corresponds to nodes in $T$.
In other words, $C'$ looks roughly like the Stone space $\hat{B}$ of the algebra $B$, the only difference is that every isolated point in $\hat{B}$ is replaced with a copy of the unit interval. 

\subsection{The reduction operator}
Let $C$ be a closed subset of $[0,1]$. 

\begin{definition}
Define \emph{the reduct} $C^r$ of $C \subseteq_{cl} [0,1]$ to be the Polish space obtained from $C$ by replacing every intom (i.e., clopen subset $\cong [0,1]$) by an isolated point.

\end{definition}

More formally, we introduce a closed equivalence relation $E$ on $C$ as follows: $x \,E\, y$ if and only if either $x = y$ or $x,y$ belong to the same intom of $C$. Then $C^{r}$ is equal to the quotient space $C/E$. We remark that a closed quotient of a compact metrised space is itself compact and metrizable (e.g., \cite[Theorem~4.2.13 or Theorem~4.4.15]{Engelking1989}), and thus (being a metrizable compact space) is Polish.  
However, in the case of a closed \( C \subseteq [0,1] \), we believe these facts are geometrically obvious, especially when all intoms are clopen (isolated) in \( C \): simply replace each such intom with its center point.

Now let  $C= C(T)$ produced by the construction, and let $[T]$ be the Stone space of the tree $T$ that is defined by first replacing every terminal node of $T$ by an isolated infinite path, and then using the usual ultrametric 
on the resulting tree.

\begin{lemma}\label{lem:2} For any $T$, 
$(C(T)')^r \cong [T]$.
\end{lemma}

\begin{proof}
By Lemma~\ref{lem:1}, we have that $C' = C(T)'$ has no isolated points. 
The informal discussion after Lemma~\ref{lem:1} will perhaps convince the reader that Lemma~\ref{lem:2} indeed holds.
Formally,
to build a homeomorphism between the reduct $(C')^r$ and $[T]$, we
shall refer to the construction. The isomorphism on the components corresponding to $\operatorname{IntAlg}(\mathbb{Q})$ is the usual back-and-forth construction 
for Stone spaces with no isolated points (induced by the isomorphism of the dual atomless Boolean algebras). This is possible because the first sub-construction simply realises
the atomless part as the Cantor set.

In the intervals produced by the second sub-construction, and the respective clopen subspaces in $[T]$, we can
define the homeomorphism by recursion on the depth of the subtree.
For that, associate a split in the tree to the respective clopen split in the construction, ignoring the atoms. 
Further, 
 an isolated path $p\in [T]$ is mapped to the reduct of the 
respective intom. We leave the details to the reader. (Note that $\mathbf{0}''$ can perform this recursion. Indeed, we need only $\mathbf{0}'$ since we do not really need to know which elements are atoms ahead of time.)
\end{proof}

Since $[T]$ is homeomorphic to the Stone space of $B$,  we conclude that
$(C')^r \cong \widehat{B}$. In particular, $C(T) \cong C(\Gamma) $ provided that $[T] \cong [\Gamma]$.  
Since $\hat{B} \cong [T]$ is in turn fully determined by the isomorphism type of the Boolean algebra $B$,
it makes sense to write $C(B)$ instead of $C(T)$ for the closed set $C$ constructed earlier.

\begin{notation}
Write $R(C)$ to denote $(C')^r$, and write $C(B)$ for $C(T)$ with $[T] \cong \hat{B}$.
\end{notation}

We arrive at:

\begin{proposition}\label{pr:du} Let $B$ be a Boolean algebra. Then  $R(C(B)) \cong \hat{B}.$

\end{proposition}

We could view $C(B)$ as an algebra of clopen sets. It is a Boolean algebra with a unary predicate $\mathtt{in}$ that can potentially hold on $a$ only if  $a$ is an atom.
  Then every atom for which $\mathtt{in}(a)$ holds should be thought of as an intom. Indeed, we suspect  that a duality holds between 
such algebras and the closed subsets of $[0,1]$ that satisfy Property~\ref{pr:in}; we however, will not need it. The duality with Boolean algebras
induced by Proposition~\ref{pr:in} (when combined with Stone duality) will be sufficient for our purposes.

Proposition~\ref{pr:du}  can be reformulated in terms of the algebra $A= A(C)$ of clopen sets of $C$ defined as we explained in the previous paragraph. For that,  we first define $A^{a}$ to be the quotient of $A$ by the ideal generated by 
the `usual' atoms (for which the unary predicate $\mathtt{in}$ fails), and then define $r(A)$ to be the Boolean algebra $A^a$ in which the predicate $\mathtt{in}$ is ignored.
We then have:

\begin{proposition}\label{pr:du1} Let $B$ be a Boolean algebra. Then   $r(A(C(B))) \cong B$. 

\end{proposition}

The proposition applies to the transformation $C$ defined  by the construction.  (More generally, satisfying Property~\ref{pr:in} is perhaps enough, but we shall not verify this conjecture.)

\subsection{Complexity analysis} We begin with a sequence of elementary lemmas.
Throughout this subsection, we assume that $D$ is a computably compact presentation of $C(B)$.

Recall that computable compactness of $D$ can be used to uniformly produce an effective list of all clopen components in $D$, and also, to  decide whether these components intersect; in the latter case we can produce the intersection~(Lemma~\ref{lem:Dlist}).

We fix such a list. In the list, every clopen set is given by a strong index describing a tuple of closed or open balls covering the component and isolating it from the rest of the components.
\emph{We identify a clopen set from the list with this index.}

Our next task is to produce complexity estimates for various properties of such clopen sets that will be important in effectively reconstructing $B$ from $D$
(indeed, from $R(D)$).

For clopen sets $X, Y \subseteq D$, write $X \sim Y$ if $X\cap D' = Y\cap D'$.


\begin{lemma}
``$X \sim Y$'' is $\Sigma^0_2$.
\end{lemma}

\begin{proof}
We need to say that each of the $X \setminus X\cap Y$ and $Y \setminus X\cap Y$ is either empty or is a finite union of isolated points.
Being an isolated point is $\Pi^0_1$, and being empty is also $\Pi^0_1$.
\end{proof}

\begin{lemma} ``$X$ is an intom in $D$'' is $\Sigma^0_1 \& \Pi^0_1$ (i.e., it can be written as conjunction of a $\Sigma^0_1$-property and a $\Pi^0_1$-property).
\end{lemma}
\begin{proof}
Under the topological assumptions about $D$, we need to check that $X$ is connected ($\Pi^0_1$) and non-singleton ($\Sigma^0_1$). 
\end{proof}

\begin{lemma}\label{lem:intomle} ``$X\cap D'$ is an intom'' is $\Sigma^0_2$.
\end{lemma}
\begin{proof}
We need to check that $X$ splits into finitely many atoms (isolated points) and an intom. Since being an isolated point is $\Pi^0_1$,  this is a search for a $\Sigma^0_1 \& \Pi^0_1 $-property, by the previous lemma.
\end{proof}

In Boolean algebras, the unary predicate $inf$ indicates that the element generates an infinite ideal.
In the context of $D$, we say that a clopen set satisfies $inf$ if it has infnitely many non-trivial, non-identical clopen splits.

\begin{lemma}  $X'$ satisfies $inf$ in $D'$ iff $X$ satisfies $inf$ in $D$. Thus, the predicate  $inf$ is $\Pi^0_2$ in $D'$ (represented as $D /{\sim}$).
\end{lemma}

\begin{proof} If $X$ contains clopen subcomponent isomorphic to $2^{\omega}$, then $inf(X)$ clearly holds.
Otherwise, $X$ should be infinite. Since no clopen subcomponent is isomorphic to $\operatorname{IntAlg}(\omega)$, $X'$ cannot possibly contain only finitely many intoms.
Indeed, if this were the case, it would have to be equal to the disjoint union of these intoms. Before taking the derivative, this would have to be just
a finite union of isolated points and intoms; such clopen component do not satisfy $inf$.
Since $inf$ is clearly $\Pi^0_2$ (it describes the infinite process of splitting further and further), the lemma follows.
\end{proof}

\begin{lemma}  ``$X\cap D'$ is atomless'' is a $\Pi^0_2$.
\end{lemma}

\begin{proof}
This is to say that there is no clopen sub-component that is an intom (this is the negation of $\Sigma^0_2$ by Lemma~\ref{lem:intomle}) and that $inf$ holds (which is $\Pi^0_2$ by the previous lemma).
\end{proof}

\subsection{The last step of the proof}\label{subsect:last-step}

Fix $C(B) \subseteq [0,1]$ produced by the construction; it is c.e.~closed.
Aiming for a contradiction, suppose $D \cong C(B)$ is a computably compact homeomorphic copy of $C(B)$.

Note that the sequence of lemmas from the previous subsection guarantee that the following properties of clopen components can be decided using $\mathbf{0}''$:
\begin{enumerate}
\item $\sim$ (the CB-equivalence);
\item  being an intom in $C(B)' = C(B)/ {\sim}$;
\item being atomless in $C(B)'$;
\item $inf$ in $C(B)'$.
\end{enumerate}

In the notation of Proposition~\ref{pr:du1}, in the induced algebra of clopen sets $A_B = A(C(B))$,  these properties correspond to:

\begin{enumerate}
\item the congruence $\sim_a$ modulo atoms for which $\mathtt{in}$ fails;
\item  being an atom in $r(A_B) = A_B / {\sim_a}$ (ignoring predicate $\mathtt{in}$);
\item being atomless in $r(A_B)$;
\item $inf$ in $r(A_B)$.
\end{enumerate}
All these properties and relations are $\mathbf{0}''$-effective. Now, $A_B$ can be viewed as the algebra generated by the indices of clopen sets in $C(B)$.
Since $r(A_B) \cong B$ by Proposition~\ref{pr:du1}, we obtain a $\mathbf{0}''$-presentation of $(B, atom, atomless, inf)$.
By a well-known result of Thurber~\cite{ThurThesis} (in the form that it appears in \cite{knight2000}) we conclude that $B$ must have a computable presentation.
This, however, contradicts the choice of $B$, which was a c.e.~presented algebra with no computable presentation.

\

The proof of Theorem~\ref{thm:1} is complete.

\subsection{Proof of Corollary~\ref{cor:Ba}}
Our proof is similar to (and is indeed simpler than)  the proof of \cite[Theorem 1.3]{CounterExa}.

 Let \( C \) be the c.e.~closed subset of \([0,1]\) constructed above.  
We need to define a computable Banach presentation (Definition~\ref{def:basp}) of \(\mathcal{C}[C; \mathbb{R}]\), which is  the space of continuous functions \( f\colon C \to \mathbb{R} \) under pointwise addition (and the induced scalar multiplication with real coefficients) and the supremum metric
$$d(f,g) = \sup_{x \in C} |f(x)-g(x)|$$
induced by the supremum norm $\| f \| = \sup_{x \in C} f(x)$.

This is done as follows.
The computable dense subset of the domain of \(\mathcal{C}[C; \mathbb{R}]\) is given by the uniformly effective collection of piecewise linear functions with finitely many breaking points with rational coordinates, where the `breaks' occur over the special points of \( C \).  
It is clear that their maxima and minima are always achieved at one of the breaking points and are thus (uniformly) computable.  

In particular, the suprema for these functions (when viewed as functions \([0,1] \to \mathbb{R}\)) are achieved at inputs \emph{that are guaranteed to be in \( C \)}.
The functions in our dense sequence are the restrictions to \(C\) of
piecewise linear functions. It follows that the supremum over \(C\) coincides with
the supremum over \([0,1]\), and is therefore uniformly computable, making the supremum metric computable in \(\mathcal{C}[C; \mathbb{R}]\). (This is the main subtlety in the proof; see Remark~\ref{rem:rem} for a further explanation.)

Moreover, \( + \) on these linear combinations is clearly uniformly effective as well.  
It remains to observe that such functions separate points in \( C \); indeed, for any distinct \( x, y \in C \), we can always find a tooth function that is zero at \( x \) and non-zero at \( y \neq x \).  
It is well-known that this implies that such functions are dense in \(\mathcal{C}[C; \mathbb{R}]\) (by the Stone--Weierstrass Theorem). This finishes the proof of Corollary~\ref{cor:Ba}.

\begin{remark} \label{rem:rem} 
The reader may be tempted to generalise the construction above to arbitrary compact \( K \) that are computable Polish, as follows.  
For any basic open ball \( B = B(c; r) \) in \( K \), consider the generalised tooth function \[ f_B(x) = \dfrac{1}{r} \sup\{0, r - d(x,c)\}, \] and then take finite $\mathbb{Q}$-linear combinations of such functions.  
However, it is not particularly clear how the suprema of such functions should be calculated exactly.
Indeed, even if we knew that the supremum of a linear combination is achieved at a special point in \(K\), it could potentially be achieved at some point that has not yet been put into \(K\), leaving us guessing. For example, we could have \(B \subset B'\), so that the supremum of 
\(m_1 f_{B} + m_2 f_{B'}\) can potentially be achieved at the formal boundary \(\delta B' = \{y : d(c', y) = r'\}\) of the smaller ball \(B' = B(c', r')\); however, \(\delta B'\) (together with its open neighbourhood) may or may not be empty, and this is undecidable for a non-computably compact space. (The simplest example is when \(c = c'\), \(r' = r/2\), \(m_1 = 1\), and \(m_2 = -1\); but the balls do not actually have to be equicentric for this situation to occur. We can put a point at the formal boundary of $B'$ arbitrarily late, and otherwise keep only $c$ inside $B$.)

The best we can claim is that the metric is left-c.e.~(lower semicomputable), i.e., it can be approximated from below.  
It is known, however, that there exist left-c.e.~presented Banach spaces that are not linearly isometric to any computable space (this follows from~\cite{McNicholl2020}), and thus our attempted argument seems to completely break apart.

In our proof for \(C \subseteq [0,1]\), we resolve this by (essentially) specifically ensuring that the formal boundaries of the balls (ends of intervals) that we use are never empty, which is straightforward for subsets of \([0,1]\). In~\cite{CounterExa}, more work is needed to avoid a similar issue for a certain \(C \subseteq \mathbb{R}^2\). Perhaps every compact computable Polish \(K\) can be effectively embedded into the Hilbert cube in some especially nice way, making the natural presentation of the space of continuous functions on the image computable Polish. Indeed, we do not know whether \(\mathcal{C}[K, \mathbb{R}]\) has a computable Banach presentation for \emph{any} computable Polish (locally) compact \(K\). We leave this as an open problem.

Some further remarks can be found in \cite{CounterExa}.
\end{remark}

\subsection{The index set result: Proof of Theorem~\ref{thm:2}}
We prove that the index set of c.e.~subspaces of $[0,1]$ that admit 
a computably compact presentation is not arithmetical, as witnessed by subsets of $[0,1]$.

 It is not difficult to prove the following:

\begin{lemma}
The index set of computably presentable Boolean algebras among the c.e.~presented ones is not arithmetical.
\end{lemma}

\begin{proof}
 In the proof of Feiner's Theorem, one fixes an effective sequence of uniformly $\Sigma^0_{2n+4}$-predicates
$(R_n)_{n \in\omega}$ that are not uniformly $\Sigma^0_{2n+3}$. Then, the isomorphism type of the coding components,  $\operatorname{IntAlg}(\mathbb{Z}^n \mathbb{Q} +1 +\mathbb{Q})$ vs.\ $\operatorname{IntAlg}(\mathbb{Z}^{n+1} +1 +\mathbb{Q})$,
is decided based on whether $R_n(n)$ holds. The predicates can be easily chosen so that this decision procedure is not uniformly $\Sigma^0_{2n+3}$ in $n$, using the usual diagonalisation argument.
In other words,  we code into $B$ the set $S = \{n: R_n(n)\} \in \Sigma^0_{\langle 2n+4 \rangle} \setminus \Sigma^0_{\langle 2n+3 \rangle}$ in the Feiner's hierarchy of sets below $\emptyset^{(\omega)}$.
It is also shown that the Boolean algebra $B(S)$ which encodes $S \subseteq \omega$ is computably presentable iff $S \in \Sigma^0_{\langle 2n+3 \rangle}$; and it is (uniformly in the description of $S$) c.e.-presented if $S \in \Sigma^0_{\langle 2n+4 \rangle}$. A detailed  self-contained proof is given in \cite{TheBook}. 

Let $P$ be any arithmetical predicate; say, $\Sigma^0_{2m+4}$.
We define a sequence of c.e.~presented Boolean algebras $(B_i)_{i \in \omega}$ as follows. 
We modify this choice of the predicates $R_n$ by simply replacing $R_n(n)$ with $P(i) \& R_n(n)$; the predicates  $R_n$ with  $n< m$ are left unchanged.

If $P(i)$ holds, then the new predicates are identical to the previous ones, and the construction produces a c.e.~presented $B_i$ without a computable copy.
Otherwise, if $P(i)$ fails, then for all $n \geq m$ we always produce a copy of $\operatorname{IntAlg}(\mathbb{Z}^{n+1} +1 +\mathbb{Q})$. The resulting Boolean algebra will clearly be computably presentable, even
though the procedure actually builds a c.e.~copy of the algebra.

Since $m$ was arbitrary, the lemma follows.
\end{proof}

Now, observe that the proof of Theorem~\ref{thm:1} gave a uniform procedure that turns $B_i$ (in the notation of the proof above) into $C(B_i)$, a c.e.~closed subset of $[0,1]$.

To establish that the index set of spaces having a computably compact presentation is not arithmetical, it remains to prove:

\begin{lemma}\label{lem:du1}
Assume that $B_i$ has almost all components of the form $\operatorname{IntAlg}(\mathbb{Z}^{n+1} +1 +\mathbb{Q})$. Then $C(B_i)$ has a computable presentation.
\end{lemma}

\begin{proof} While $C(B_i)'$ depends only on the isomorphism type of $B_i$,
it appears that $C(B_i)$ depends on the exact presentation of $B_i$. However, this is not the case.

The proof is based on the observation that every split in the tree  $T = T_i$ of $B = B_i$ is always accompanied with at least one extra `junk' atom.

Recall that the predicate $inf$ indicates that an element $x$ generates an infinite ideal in $B$. Here we also use notations from Section~\ref{subsect:last-step}. The introduction of `junk' atoms ensures the following property: if $x\in B$ satisfies $inf(x)$, then $x$ must contain infinitely many atoms $y$ such that $\texttt{in}(y)$ holds (i.e., atoms viewed as intoms) \text{and} infinitely many atoms $z$ satisfying $\neg \texttt{in}(z)$. Indeed, consider the node $\sigma \in T$ that corresponds to $x\in B$. If $inf(x)$ is true, then (1)~we have infinitely many paths $p\in[T]$ that go through $\sigma$---these paths turn into intoms in the construction; and (2)~the subtree of $T$ which has root $\sigma$ is infinite, hence, this subtree has infinitely many splits and produces infinitely many `junk' atoms.

We deduce the following property: for any $x\in B$,
$$
	\mathrm{card}(\{ z\leq x : \texttt{in}(z)\}) < \omega\ \Leftrightarrow\ \mathrm{card}(\{ z\leq x : \neg \texttt{in}(z)\}) < \omega.
$$
Then Proposition~\ref{prop:Cenzer-Remmel} stated below implies that the isomorphism type of the structure $(B, \texttt{in})$ does not depend on the choice of a presentation of the structure $B$. 

In order to finish the proof, we observe that up to homeomorphism, the space $C(B)$ can be constructed as follows. Choose an appropriate oracle $X\subseteq \omega$ and construct a $\Pi^0_1(X)$-class $[T]$ such that $[T] \cong \widehat{B}$ and $T$ has additional $X$-computable labels: a node $\sigma \in T$ has label $\texttt{in}(\sigma)$ if and only if $\sigma$ isolates an atom $p\in [T]$ satisfying $\texttt{in}(p)$ in $(B,\texttt{in})$. Observe that the homeomorphism type of the labelled $[T]$ depends only on the isomorphism type of $(B,\texttt{in})$. Consider the standard homeomorphic embedding $\Psi$ from $[T]$ into $[0,1]$, and in the image $\Psi([T])$ replace every $\Psi(p)$ satisfying $\texttt{in}(p)$ with a (small enough) clopen component $[a_p,b_p]$. Then the modified $\Psi([T])$ is homeomorphic to $C(B)$.
\end{proof}

Let $B$ be a Boolean algebra. For a set $U \subseteq \operatorname{Atom}(B)$, we define the function $c(\cdot; U)$ as follows: for $x\in B$, put
\[
	c(x; U) = \mathrm{card}(\{ z \in U : z\leq x\}).
\]

\begin{proposition}\label{prop:Cenzer-Remmel}
	Let $B_0$ and $B_1$ be countable, atomic Boolean algebras, and let $U_i\subseteq \operatorname{Atom}(B_i)$ be infinite sets with the following property:
	\begin{itemize}
		\item[($\dagger$)] $\forall x [ c(x;U_i) \text{ is finite}\ \Leftrightarrow\ c(x; \mathrm{Atom}(B_i) - U_i) \text{ is finite}]$.	
	\end{itemize} 
	Consider the congruence $\sim_{i}$ on $B_i$ which is defined as follows: $x\sim_{i} y$ if and only if the symmetric difference $(x-y)\cup (y-x)$ is a finite sum of elements from $U_i$. If the quotient Boolean algebras $B_0/ {\sim_{0}}$ and $B_1/ {\sim_{1}}$ are isomorphic, then the structures $(B_0,U_0)$ and $(B_1,U_1)$ are also isomorphic.
\end{proposition}
\begin{proof}
	Vaught~\cite{Vau} obtained a sufficient condition for the existence of isomorphism between two countable Boolean algebras $B_0$ and $B_1$. Here we follow the exposition of the Vaught's result given in Section~1.5 of~\cite{GonBoo}. In order to prove that $B_0 \cong B_1$, it is sufficient to construct a set $S \subseteq B_0 \times B_1$ satisfying the following axioms:
	\begin{itemize}
		\item[(a)] $(0_{B_0}, 0_{B_1})$ and $(1_{B_0}, 1_{B_1})$ belong to $S$;
		
		\item[(b)] if $(0_{B_0},y) \in S$, then $y=0_{B_1}$;
		
		\item[(c)] if $(x,0_{B_1}) \in S$, then $x=0_{B_0}$;
		
		\item[(d)] if $(x,y)\in S$ and $a \leq x$, then there exists $b \leq y$ such that $(a,b)$ and $(x-a, y-b)$ belong to $S$;
		
		\item[(e)] if $(x,y) \in S$ and $b \leq y$, then there exists $a \leq x$ such that $(a,b), (x-a, y-b) \in S$.
	\end{itemize}
	
	For the sake of convenience, we use the following notations: for $x\in B_i$, put $c^{+}(x) = c(x; U_i)$ and $c^{-}(x) = c(x; \operatorname{Atom}(B_i) - U_i)$. 
	
	Let $F$ be an isomorphism from $B_0/{\sim_{0}}$ onto $B_1/{\sim_{1}}$. Notice the following: if $F(x/{\sim_0}) = y/{\sim_1}$, then we have $c^{-}(x) = c^{-}(y)$.
	
	We define the set
	\begin{equation*}
		S = \{ (x,y) : F(x/{\sim_0}) = y/{\sim_1} \text{ and } c^{+}(x) = c^{+}(y) \}.
	\end{equation*}
	It is clear that the set $S$ satisfies Axioms~(a)--(c). 
	
	We check Axiom~(d). Suppose that $(x,y) \in S$ and $a \leq x$. If we have $c^{-}(a) = c^{-}(x-a) = \omega$, then we can choose any $b\in B_1$ such that $b< y$ and $F(a/{\sim_0}) = b/{\sim_1}$. (Indeed, by Property~($\dagger$), we have $c^{+}(a) = c^{-}(a) = \omega$ and $c^{+}(b) = c^{-}(b) = c^{-}(a) = \omega$.)
	
	If $c^{-}(a) < \omega$, then we take $b\in B_1$ such that $b\leq y$, $F(a/{\sim_0}) = b/{\sim_1}$, and $c^{+}(b) = c^{+}(a)$. Such an element $b$ can be always chosen, since $c^{+}(y) = c^{+}(x)$. We observe that 
	\[
		c^{+}(y-b) = c^{+}(x-a) = \begin{cases}
			\omega, & \text{if } c^{+}(x) = \omega,\\
			c^{+}(x) - c^{+}(a), & \text{if } c^{+}(x) < \omega.
		\end{cases}
	\]
	A similar argument works for the case when $c^{-}(x-a) < \omega$. We conclude that $(a,b),(x-a,y-b) \in S$ and that $S$ satisfies Axiom~(d). Verification of Axiom~(e) goes mutatis mutandis.
	
	By the Vaught's theorem, we obtain that there exists an isomorphism $G$ between Boolean algebras $B_0$ and $B_1$. An analysis of the construction of the isomorphism $G$ (see Theorem~1.5.2 in~\cite{GonBoo}) shows that every atom $z\in \operatorname{Atom}(B_0)$ has the following property: $(z, G(z)) \in S$. The definition of $S$ implies that $c^{+}(z) = c^{+}(G(z))$, and this means that $z\in U_0$ if and only if $G(z) \in U_1$. Therefore, the map $G$ is an isomorphism from the structure $(B_0,U_0)$ onto $(B_1,U_1)$. 
\end{proof}

This concludes the proof of Theorem~\ref{thm:2}.

\section*{Funding}
This research is funded by the Science Committee of the Ministry of Science and Higher Education of the Republic of Kazakhstan (Grant No.~AP19676989).
The work of Bazhenov and Goncharov was also carried out within the framework of the state contract of the Sobolev Institute of Mathematics (project no. FWNF-2022-0011).
The work of Bazhenov was also supported by Nazarbayev University Faculty Development Competitive Research Grants 201223FD8823.

\section*{Acknowledgements}
The authors are grateful to Alibek Iskakov for assistance with the diagrams.

\bibliographystyle{alpha}
\bibliography{mainbib}

\end{document}